\documentclass[11pt,sumlimits,nointlimits,namelimits,reqno]{amsart}
\usepackage{amsopn, amsfonts,amssymb, amsbsy, amscd,graphicx,tabularx,oldgerm,appendix,amsmath,mathdots}
\usepackage[mathscr]{eucal}
\usepackage{color,pstricks}
\usepackage[all]{xy}
\usepackage{comment}
\usepackage[utf8]{inputenc}

\newcounter{thmnum}
\newtheorem {proposition}[thmnum]{Proposition}

\newtheorem{corollary}[thmnum]{Corollary}
\newtheorem {theorem}[thmnum]{Theorem}
\newtheorem {lemma}[thmnum]{Lemma}

\theoremstyle{definition}

\theoremstyle{remark}

\newcommand{\T}{{\mathbb T}}
\newcommand{\Q}{{\mathbb Q}}
\newcommand{\Z}{{\mathbb Z}}
\newcommand{\C}{{\mathbb C}}
\newcommand{\R}{{\mathbb R}}

\newcommand{\F}{{\mathbb F}}
\newcommand{\Fqt}{{ \Fq((\frac{1}{t})) }}
\newcommand{\E}{{\mathbb E}}

\newcommand{\ba}{{\mathbf a}}
\newcommand{\bchi}{{\boldsymbol{\chi}}}

\newcommand{\calI}{{\mathcal I}}

\newcommand{\calL}{{\mathcal L}}
\newcommand{\Fq}{{\mathbb F_{q}}}

\def\Fqn{\F_q^n}
\def\rk{\text{rk }}
\def\span{\text{span}}
\newcommand{\nor}[1]{\left\lVert #1 \right\rVert}

\newcommand{\abs}[1]{\left\lvert #1 \right\rvert}

\begin{document}
\title{Linear and quadratic uniformity of the M\"obius function over $\Fq[t]$}

\author[P.-Y. Bienvenu]{Pierre-Yves Bienvenu}
\address{P.-Y. Bienvenu,  Institut Camille-Jordan, Université Lyon 1, 43 boulevard du 11 novembre 1918
69622 Villeurbanne cedex, France}
\email{pbienvenu@math.univ-lyon1.fr}

\author[T. H. L\^e]{Th\'ai Ho\`ang L\^e}
\address{T. H. L\^e, Department of Mathematics, The University of Mississippi, University, MS 38677, United States}
\email{leth@olemiss.edu}

\begin{abstract}
We examine correlations of the M\"obius function over $\F_q[t]$  with linear or quadratic phases,
that is, averages of the form
\begin{equation}
\label{eq:average}
\frac{1}{q^n}\sum_{\deg f<n} \mu(f)\chi(Q(f))
\end{equation}
for an additive character $\chi$ over $\F_q$ and
a polynomial $Q\in\F_q[x_0,\ldots,x_{n-1}]$ of degree at most 2 in the coefficients $x_0,\ldots,
x_{n-1}$ of $f=\sum_{i< n}x_i t^i$.
Like in the integers, it is
reasonable to expect that, due to the random-like behaviour of
$\mu$,
such sums should exhibit considerable cancellation.
In this paper we show that the correlation \eqref{eq:average}
is bounded by $O_\epsilon \left( q^{(-\frac{1}{4}+\epsilon)n} \right)$ for any $\epsilon >0$ if $Q$ is linear
and $O \left( q^{-n^c} \right)$ for some absolute constant $c>0$
if $Q$ is quadratic.
The latter bound may be 
reduced to $O(q^{-c'n}$) for some $c'>0$ when $Q(f)$
is a linear form in the coefficients of $f^2$, that is, a Hankel quadratic form, whereas for general quadratic forms, it relies on a bilinear version of the additive-combinatorial Bogolyubov theorem.
%
\end{abstract}

\maketitle
\section{Introduction}
Let $p$ be a prime and $q=p^s$ a prime power ($s\geq 1$).
Let $\Fq$ be the field over $q$ elements and $\Fq[t]$ be the ring of polynomials over $\Fq$.
The M\"obius function on $\Fq[t]$ is defined, like its counterpart in the integers, by
\[
\mu(f)=
\left\{
  \begin{array}{ll}
    (-1)^k & \hbox{where $k$ is the number of monic irreducible factors of $f$, if $f$ is squarefree,} \\
    0 & \hbox{otherwise.}
  \end{array}
\right.
\]
In the integers, a folklore conjecture predicts
that $\mu$ is so random-like that it does not correlate with
any bounded ``reasonable'' or ``low-complexity'' function $F$,
in the sense that
\begin{equation}
\label{eq:nonCorr}
\sum_{n\leq x}\mu(n)F(n)=o(x).
\end{equation}
For instance, linear or quadratic phases, that is, functions $F$
defined by $n \mapsto e(\alpha n)$
or $n\mapsto e(\alpha n^2)$ should satisfy equation \eqref{eq:nonCorr}. Davenport \cite{Dav} proved such a statement
 for
linear phases and
Green and Tao for general nilsequences \cite{gt,MN}.  We do not attempt to define nilsequences here, but note that they include sequences formed by regular polynomials such as $F(n)=e\left(\alpha n^2 + \beta n + \gamma\right)$, as well as ``generalized polynomials'' such as $F(n)=e \left( \lfloor n \alpha \rfloor n \beta \right)$.
Together with Green and Tao's work \cite{InvU3} on the inverse theorem for the Gowers $U^3$ norm, this implies that $\nor{\mu}_{U^3(N)}=o_{N\rightarrow +\infty}(1)$.

In this paper we examine similar correlations
over $\F_q[t]$, that is,
we aim to show that
$$
\sum_{\deg f<n}\mu(f)F(f)=o(q^n)
$$
for ``reasonable'' functions $F$. Quadratic and linear phases correspond to functions of the form 
$\chi(Q(f))
$ 
for an additive character $\chi$ over $\F_q$ and
a polynomial $Q\in\F_q[x_0,\ldots,x_{n-1}]$ of degree at most 2 in the coefficients $(x_0,\ldots,
x_{n-1})\in\Fqn$ of $f=\sum_{i< n}x_i t^i$. 
Recall that the group $\widehat{\F_q}$ of additive characters is
isomorphic to (the additive group of) $\F_q$.
To express the isomorphism,
let $\mathrm{Tr}: \Fq \rightarrow \F_{p}$ be the trace map. For $a \in \Fq$, let us denote 
\[
e_{q}(a)=\exp \left( \frac{2 \pi i \mathrm{Tr}(a)}{p} \right).
\]
Then the isomorphism
$\F_q\rightarrow\widehat{\F_q}$ is given by $r\mapsto \chi_r$ where for any $r\in\F_q$, the character $\chi_r$
is defined by $\chi_r(x) = e_q(rx)$.

We now state our main results.
\begin{theorem} 
\label{trm:linforms}
For any $\epsilon >0$ and $\chi\in\widehat{\F_q}$, for any linear form $\ell \in \Fq[x_0, \ldots, x_{n-1}]$, we have
\begin{equation} \label{eq:th1}
\sum_{\deg f<n} \mu(f)\chi(\ell(f))\ll_{\epsilon, q} q^{(3/4+\epsilon)n}.
\end{equation}
uniformly in $n$ and $\ell$. 
\end{theorem}
It suffices to prove Theorem \ref{trm:linforms} for $\chi=\chi_1$.
In the integer case, Davenport \cite{Dav} showed that for any $A>0$, we have
\[
\sum_{n=1}^{N} \mu(n) e(n \alpha) \ll_A N (\log N)^{-A}
\]
uniformly in $\alpha \in \R/\Z$, where the implied constant is ineffective due to the possible existence of Siegel zeroes. Under the Generalized Riemann Hypothesis (GRH), the best result is due to  
Baker-Harman \cite{baker-harman} and Montgomery-Vaughan (unpublished), who showed that for any $\epsilon >0$,
\begin{equation} \label{eq:bh}
\sum_{n=1}^{N} \mu(n) e(n \alpha) \ll_\epsilon N^{3/4 + \epsilon}
\end{equation}
uniformly in $\alpha \in \R/\Z$. Our exponent $\frac{3}{4} + \epsilon$ in \eqref{eq:th1} matches the one in \eqref{eq:bh} (though it is reasonable to conjecture that in both cases the best exponent is $\frac{1}{2} + \epsilon$). However, our proof of \eqref{eq:th1} differs from that of \eqref{eq:bh} in some respects. This is due to the fact that the topologies of $\Fq[t]$ and $\Z$ are different and some standard tools in $\Z$ such as summation by parts do not have counterparts over $\Fq[t]$. In particular, our proof of \eqref{eq:th1} uses $L$-functions of \textit{arithmetically distributed relations} introduced by Hayes \cite{hayes}, as opposed to Dirichlet $L$-functions. We remark that very recently and independently of us, Sam Porritt \cite{porritt} has proved a result similar to Theorem \ref{trm:linforms}.

Regarding quadratic polynomials, we have the following similar, but weaker, result.
It depends on 
 the polylogarithmic bilinear Bogolyubov theorem \cite[Theorem 1.3]{LovettHosseini},
a quantitative improvement of a structural result
in additive combinatorics, the bilinear Bogolyubov theorem from our companion paper \cite{bilinBogo}. We introduce this theorem in Section \ref{sec:bogo}.
\begin{theorem} 
\label{trm:quadforms}
Assume $p>2$. There exists a constant $c>0$ such that the following holds. For any $\chi\in\widehat{\F_q}$,
we have
\begin{equation} \label{eq:th2}
\frac{1}{q^n}\sum_{\deg f<n} \mu(f)\chi(Q(f))\ll_{q} q^{-n^c}
\end{equation}
uniformly in $n$ and the quadratic polynomial $Q$ in $\F_q[x_0, \ldots, x_{n-1}]$.\end{theorem}

Note that the quality of \eqref{eq:th2} is superior to 
Green and Tao's bound for nilsequences in \cite{MN}, namely that if ${s(n)}$ is a nilsequence, then for any $A >0$, one has 
\begin{equation} \label{eq:gt}
\sum_{n=1}^N \mu(n) s(n) \ll_{s,A} N \log^{-A} N.
\end{equation} 

We have another result for quadratic phases similar to $n \mapsto e(\alpha n^2 + \beta n)$. In this case, our bound is easier to prove and gives a polynomial saving. We need some extra notations to state our result (see Section \ref{sec:notation} for more precise definitions). On $\Fq[t]$ there is a natural norm $|f| = q^{\deg f}$. The completion of $\Fq[t]$ with respect to this norm is $\Fqt$, the ring of formal Laurent series in $1/t$. On $\Fq((\frac{1}{t}))$, we define the additive character $e(\alpha)=e_{q}((\alpha)_{-1})$, where $(\alpha)_{-1}$ denotes the coefficient of $t^{-1}$ in $\alpha$. 

\begin{theorem}
\label{trm:Hankel} There exists a constant $\epsilon > 0$ (independent of $q$)
such that
\begin{equation} \label{eq:th3}
\sum_{\deg f<n} \mu(f) e( \alpha f^2 + \beta f) \ll_{q} q^{(1-\epsilon)n}
\end{equation}
uniformly in $n$ and $\alpha, \beta \in \Fqt$.
\end{theorem}

Note that we do not require $p>2$ in Theorem \ref{trm:Hankel} since when $p=2$ the map  $f \mapsto  (\alpha f^2 + \beta f)_{-1}$ is linear and Theorem \ref{trm:Hankel} follows from Theorem \ref{trm:linforms}. When $p$ is odd, the symmetric matrix of the quadratic form $f \mapsto (\alpha f^2)_{-1}$ is a \textit{Hankel matrix}, i.e., a matrix whose $(i,j)$-entry depends only on $i+j$. Thus Theorem \ref{trm:Hankel} can be reformulated in terms of Hankel matrices alone. We remark that in the integers, under GRH we have bounds with polynomial savings for the sum $\sum_{n=1}^N \mu(n) e(\alpha n^k)$ (see \cite{hh,zl}). 

We point out that the motivation to tackle correlations with quadratic phases, as for the corresponding result in the integers,
is the derivation of Gowers norms estimates $\nor{\mu}_{U^3(\F_q^n)}=o(1)$, where the set of polynomials of degree less than $n$ is identified with $\Fqn$.
We refrain from definining Gowers norms here and refer instead to \cite{InvFF} for a general theory, but we highlight that
the bound $\nor{\mu}_{U^3(\F_q^n)}=o(1)$
allows one to control various linear autocorrelations of $\mu$; for instance, it implies that
$$
\sum_{\deg f,\deg g <n}\mu(f)\mu(f+g)\mu(f+2g)\mu(f+3g)=o(q^{2n}).
$$
For $p>2$, it was shown by Green and Tao \cite{InvU3} that the norm
$\nor{\cdot}_{U^3(\F_p^n)}$ is controlled by
correlations with quadratic polynomials\footnote{This actually holds for $p=2$ as well thanks to a theorem of Samorodnitsky \cite{Samorodnitsky}.}.

  However, Theorem \ref{trm:quadforms} only yields
a Gowers norm estimate when $q=p$ is a prime. To see this,
fix a group isomorphism  $\phi:\F_q\rightarrow\F_p^{s}$
and let $\phi_n:\F_q^n\rightarrow\F_p^{sn}$
be the group isomorphism it induces in dimension $n$. For $f\in\F_q^n$, write $\tilde{f}=\phi_n(f)$.
Observe that not any $\F_p$-quadratic form $P(\tilde{f})$
can be realised as $\mathrm{Tr}(Q(f))$ for some $\F_q$-quadratic form
$Q(f)$; this can be seen by simple counting.  
But controlling $\nor{\mu}_{U^3(\F_q^n)}$ precisely requires  control
of correlations of $\mu$ with any
$\F_p$-quadratic form $P(\tilde{f})$,
whereas
 Theorem \ref{trm:quadforms}
 only deals with $\F_q$-quadratic forms. 

The organization of the paper is as follows. In Section \ref{sec:prelim} we collect necessary facts that will be used in the proofs; in particular we introduce and motivate Hayes' theory as well as the bilinear Bogolyubov theorem (Theorem \ref{conj:polyLog}). In Section \ref{sec:char} we prove a character sum estimate, using standard complex analysis as well as Hayes' theory, and exploit it to infer Theorem \ref{trm:linforms} in Section \ref{sec:exponential}. In Section \ref{sec:quad} we use Vaughan's identity to reduce Theorem \ref{trm:quadforms} to a problem in bilinear and quadratic algebra and 
prove it in Section \ref{useBB} using Theorem \ref{conj:polyLog}.
Finally, we derive the bound \eqref{eq:th2}
for the Hankel case in Section \ref{Hankel}, that is, Theorem \ref{trm:Hankel}.

\subsection*{Acknowledgements}
The authors would like to thank Ben Green for suggesting to prove and use a bilinear Bogolyubov
theorem, Terence Tao for helpful conversations and Trevor Wooley for suggesting the Hankel case.
We are also grateful to Sam Porritt for drawing our attention to \cite{porritt} in which results similar to our linear case are proved, and for useful discussions which lead to an improvement of an earlier version of Theorem \ref{trm:linforms}. The first author is thankful to
his supervisor Julia Wolf
for guidance. 

Part of this work was carried out while the first author was staying at the Simons Institute for the theory of computing and supported by a travel grant of the University of Bristol Alumni Foundation. He also benefited from the hospitality of  the University of Mississippi. The second author was supported by National Science Foundation Grant DMS-1702296 and a Ralph E. Powe Junior Faculty Enhancement Award from Oak Ridge Associated Universities.

\section{Preliminaries} \label{sec:prelim}

\subsection{Notation and basic facts} \label{sec:notation}
A useful reference for the circle method in function fields,
of which the basics are sketched below, is \cite{WooleyLiu}.
Let $\Fq(t)$ be the field of fractions of $\Fq[t]$. On $\Fq(t)$ we can define a norm by $|f/g|=q^{\deg f - \deg g}$, with the convention $\deg 0 =-\infty.$ The completion of $\Fq(t)$ with respect to this norm is 
\[
\Fq \left( \left(\frac{1}{t}\right) \right)=\left\{ \alpha = \sum_{i=-\infty}^{n}a_{i}t^{i}: n \in \Z,  a_{i} \in \Fq \textrm{ for every } i \right\},
\]
the set of formal Laurent series in $\frac{1}{t}$. It is easy to see that if $\alpha$ is as above and $a_n \neq 0$ then $|\alpha| = q^{n}$.

Then $\Fq[t] \subset \Fq(t) \subset \Fq((\frac{1}{t}))$, and $\Fq[t],\Fq(t)$ and $\Fq((\frac{1}{t}))$ are the analogs of $\Z, \Q, \R$ respectively. 

Let us put $\T=\{ \alpha \in \Fq((\frac{1}{t})): |\alpha|<1\}$. This is analogous to the usual torus $\R/\Z$. Let $\mathrm{Tr}: \Fq \rightarrow \F_{p}$ be the trace map. For $a \in \Fq$, let us denote 
\[
e_{q}(a)=\exp \left( \frac{2 \pi i \mathrm{Tr}(a)}{p} \right).
\]
This is an additive character on $\Fq$. All additive characters on $\Fq$ are given by $a \mapsto e_q(ra)$ for some $r \in \Fq$.

For $\alpha \in \Fq((\frac{1}{t}))$, we write $(\alpha)_{-1}$ to denote the coefficient of $t^{-1}$ in $\alpha$. 
We define $e(\alpha)=e_{q}((\alpha)_{-1})$. This is an additive character on $\Fqt$ and allows us to do Fourier analysis on $\Fq[t]$. It is analogous to the function $x \mapsto e^{2\pi i x}$ with a few differences. For example, $e(\alpha) =1$ does \textit{not} imply that $\alpha \in \Fq[t]$. All additive characters on $\Fq[t]$ are given by $f \mapsto e(f \alpha)$ for some $\alpha \in \T$.

We denote by $M$ the set of all monic polynomials in $\Fq[t]$, $A_n$ the set of all polynomials of degree $n$ which are monic, $G_n$ the set of all polynomials (not necessarily monic) of degree less than $n$ and $\calI$ the set of all monic, irreducible polynomials. We use the convention that $\sum_{\deg f=l}$ means $\sum_{f\in A_l}$ (that is, a sum over monic polynomials).

The von Mangoldt function on $\Fq[t]$ is defined by
\[
\Lambda(f)=
\left\{
  \begin{array}{ll}
    \deg P, & \hbox{if $f = P^k$ for some monic irreducible $P$ and $k \geq 1$,} \\
    0, & \hbox{otherwise.}
  \end{array}
\right.
\]
Recall that the ``prime number theorem'' on $\Fq[t]$ reads
$$
\sum_{\deg f=l}\Lambda(f)=q^l.
$$
\subsection{$L$-functions of arithmetically distributed relations} \label{sec:hayes}
To prove Theorem \ref{trm:linforms}, 
we first observe that any linear form on $G_n$
can be represented as a map $f\mapsto (\alpha f)_{-1}$
for some $\alpha\in\T$.
Thus Theorem \ref{trm:linforms}
can be rephrased as a bound for sums of the form
$$
\sum_{f\in G_n}\mu(f)e(\alpha f)
$$
or, equivalently and more conveniently, of the form
$$
\sum_{f\in A_n}\mu(f)e(\alpha f).
$$
Now if $\alpha$ is approximated by a fraction $a/Q$ of polynomials up to a remainder $\beta=\sum_{i=-\infty}^{-l}\beta_i t^i$ for some $l\geq 2$, that is, $\alpha=a/Q+\beta$, then
$e(\alpha f)=e(af/Q)e(\beta f)$ depends only on the residue class of $f$ modulo $Q$
and the coefficients of the terms of degrees at least $l-1$ of $f=\sum_{i=1}^n a_i t^{n-i}+t^n$. We refer to $a_{1}, \ldots, a_{l}$ as the first $l$ coefficients of $f$ (if $i>n$ then we define $a_i = 0$).
We thus need to understand functions on $A_n$
that only depend on the congruence class modulo a fixed
modulus $Q$ and the first $l$ coefficients. 
Hence for $l \geq 0$, $Q \in \Fq[t]$, we define an equivalence relation $R_{l,Q}$ on $M$ as follows:
\[
 f \equiv g \pmod{R_{l,Q}} \textup{ if } f \equiv g \pmod{Q} \textup{ and the first $l$ coefficients of $f$ and $g$ are the same.}
\]
It is an example of an \emph{arithmetically distributed relation},
of which Hayes \cite[Section 8]{hayes} developed the theory,
which we briefly review. The relevant facts can also be found in \cite{hsu} or \cite{car}.

It is easy to check that $M / R_{l,Q}$ is a semigroup with respect to multiplication on $\Fq[t]$. The equivalence class of a polynomial $f \in \Fq[t]$ is invertible in $M / R_{l,Q}$ if and only if $(f, Q) =1$. 

Put $G_{l,Q}: = (M / R_{l,Q})^{\times}$, the set of invertible elements. This is a group of cardinality $q^{l} \phi(Q)$, where $\phi(Q) = \# (\Fq[t] / (Q))^{\times}$. 
Note that $G_{0,Q}$ is simply $(\Fq[t] / (Q) )^{\times}$. 

For a character $\lambda$ on $G_{l,Q}$, we extend it to all of $M$ by setting $\lambda(f) =0$ if $(f, Q) \neq 1$. 
We define the $L$-function associated with $\lambda$ as 
\[
L(s, \lambda) = \sum_{f \in M} \lambda(f) \frac{1}{|f|^s}
\]
which converges absolutely for $\Re(s) >1$. It is convenient to put 
\begin{equation} \label{eq:defl-1}
\calL(z, \lambda) = \sum_{f \in M} \lambda(f) z^{\deg(f)} = \sum_{n=1}^\infty z^n \sum_{f \in A_n} \lambda(f).
\end{equation}
Then $L(s, \lambda) = \calL(q^{-s}, \lambda)$. We have the Euler product formula
\begin{equation} \label{eq:euler}
 \calL(z,\lambda)=\prod_{P \in \calI} \left( 1 - \lambda(P) z^{\deg P} \right)^{-1}
\end{equation}
for $|z| <1/q$.

In the same range of $z$, we also have
\begin{equation} \label{eq:euler2}
 \frac{1}{\calL(z,\lambda)}=\prod_{P} \left( 1 - \lambda(P) z^{\deg P} \right) = \sum_{f \in M} \mu(f) \lambda(f) z^{\deg f} = \sum_{n=1}^\infty z^n \sum_{f \in A_n} \mu(f) \lambda(f).
\end{equation}
 
The character constantly equal to 1 on $G_{l,Q}$ is called
the \emph{principal character}.
When $\lambda$ is not the principal character, $\calL(z, \lambda)$ is a polynomial of degree $d(\lambda) < l + \deg Q$ (\cite[Lemma 8.2]{hayes}). The \textit{Generalized Riemann Hypothesis} states that all roots of $\calL(z,\lambda)$ have modulus $q^{-1/2}$ or 1 for any character $\lambda$ modulo
an arithmetically distributed congruence relation such as $R_{l,Q}$. Weil's proof (for Dirichlet characters) was extended to these generalised characters by Rhin \cite{Rhin} (see in particular Chapitre 2, section 4 to 6). 
In other words, we can write
\begin{equation} \label{eq:weil}
 \calL(z, \lambda) = \prod_{i=1}^{d(\lambda)} (1 - \alpha_i z)
\end{equation}
where $|\alpha_i|=q^{1/2}$ or 1 for $i =1, \ldots, d(\lambda)$. In particular, $\calL(z, \lambda)$ can be extended to an entire function and \eqref{eq:euler2} remains valid when $|z|<q^{-1/2}$.

When $\lambda$ is the principal character of $G_{l,Q}$, we have 
\begin{eqnarray*} 
 \calL(z,\lambda) &=& \prod_{\substack{P \in \calI,\\ (P, Q) =1}} \left( 1 - z^{\deg P} \right)^{-1} \\ 
&=& \prod_{\substack{P \in \calI,\\ P|Q}} \left( 1 - z^{\deg P} \right) \prod_{P \in \calI} \left( 1 - z^{\deg P} \right)^{-1} \\
&=& \prod_{\substack{P \in \calI,\\ P|Q}} \left( 1 -  z^{\deg P} \right) \frac{1}{1-qz}.
\end{eqnarray*}
Consequently,  $\calL(z, \lambda)$ can be extended to a meromorphic function and 
\begin{equation} \label{eq:principal}
 \frac{1}{\calL(z,\lambda)} = \sum_{n=1}^\infty z^n \sum_{f \in A_n, (f,Q)=1} \mu(f) = (1-qz) \prod_{\substack{P \in \calI,\\ P|Q}} \left( 1 - z^{\deg P} \right)^{-1}
\end{equation}
for all $|z| \neq 1$.

\subsection{The Bilinear Bogolyubov Theorem} \label{sec:bogo} When proving Theorem \ref{trm:quadforms}, we will suppose for a contradiction that  
\[
\sum_{f \in G_n} \mu(f) \chi(Q(f)) \geq \epsilon q^n.
\]

Let $M$ be the $n\times n$ symmetric matrix corresponding to $Q$ and $k$ an integer. For any $a \in G_{k+1}$,
consider the map $L_a:G_{n-k}\rightarrow G_n$ that maps $f$ to $af$. We also write $L_a$ to denote its $n\times (n-k)$ coordinate matrix
in the canonical basis (i.e., the basis of monomials). For any $(a,b)\in G_{k+1}^2$, let
$M_{a,b}=L_a^T M L_b + L_b^TML_a$; it is a symmetric $(n-k)\times (n-k)$ matrix.

After exploiting Vaughan's identity in Section \ref{Vaughan},
we will find that for some $n\ll k\leq n$, $M$ has the property that the set of pairs
$$
P_{h}:=\{(a,b)\in G_{k+1} \times G_{k+1} \vert \rk M_{a,b}\leq h \}
$$
is large, that is, it contains at least $\delta q^{2k+2}$ pairs for
some parameters $\delta$ and $h$ (depending on $\epsilon$ and $n$).
We will want to convert this information
about the ranks of many $M_{a,b}$ into
one on the rank of $M$ itself.
However, we need these pairs to have some special structure in
order to extract some information, in particular it would be
extremely convenient if the set 
\begin{equation}
\label{eq:monomes}
\{(t^i,t^j) \mid (i,j)\in\{0,\ldots,k\}^2\}
\end{equation} could be in $P_h$,
because $M_{t^i,t^j}$ is then simply a submatrix of $M$.
Unfortunately, its large size alone does not force $P_{h}$ to contain
such a nice structure, but
to boost our chances, we are ready to do
some additive smoothing, that is, adjoining to our set $P$
elements such as $(a_1-a_2,b)$ whenever $(a_1,b)$ and $(a_2,b)$
are in $P$; and the same on the second coordinate.
The rank remains controlled under this operation, because $\rk M_{a_1-a_2,b}= \rk (M_{a_1,b}-M_{a_2,b})\leq 2h$. Now our companion paper \cite{bilinBogo}
shows that additive smoothing does indeed produce useful structures.
Here is the result we get \cite[Corollary 4]{bilinBogo}.
\begin{proposition}
\label{prop:bilinBogo}
For any $\delta$, there exists a constant $c(\delta)$
such that the following holds.
If $\abs{P_{h}}\geq \delta q^{2k+2}$,
then there exist $\F_p$-subspaces $W_1,W_2$ of the $\F_p$-vector space $G_{k+1}$ of codimension
$r_1,r_2$ and
$\F_p$-bilinear forms $Q_1,\ldots,Q_r$ on $W_1\times W_2$
such that
$P_{64h}=\{(a,b)\in G_{k+1}^2\vert \text{\emph{rk} } M_{a,b}\leq 64h\}$ contains the set
\begin{equation*}
\{(x,y)\in W_1\times W_2\mid Q_1(x,y)=\cdots=Q_r(x,y)=0\}
\end{equation*}
and
$\max(r,r_1,r_2)\leq c(\delta)$.
\end{proposition}
We call this statement a bilinear Bogolyubov theorem, by analogy with
the original (linear) Bogolyubov theorem.
We found that we can take $c(\delta)$ to be $O(\exp(\exp(\exp(\log^{O(1)}1/\delta))))$ where the implied constants may depend on $q$,
but unfortunately, because $\delta$ will be as small as, say, $n^{-5}$, this bound for $c(\delta)$ is too large.
By analogy with Sanders' bound for the linear Bogolyubov theorem \cite{Sanders},
it is reasonable to imagine \cite[Conjecture 3]{bilinBogo} that the linear and bilinear codimensions $r,r_1,r_2$ could be taken as small as polylogarithmic in $\delta^{-1}$. In \cite{bilinBogo} we show that indeed we can take $r$ and one of $r_1$ and $r_2$ to be polylogarithmic in $\delta^{-1}$.
Recently, Hosseini and Lovett \cite[Theorem 1.3]{LovettHosseini} lowered $c(\delta)$ to 
$\log ^{O(1)}\delta^{-1}$, at the cost of replacing 64 in Proposition \ref{prop:bilinBogo} by a larger constant.
\begin{theorem}[polylogarithmic bilinear Bogolyubov]
\label{conj:polyLog}
For any $\delta$, if $\abs{P_{h}}\geq \delta q^{2k+2}$,
then there exist $\F_p$-subspaces $W_1,W_2$ of the $\F_p$-vector space $G_{k+1}$ of codimension
$r_1,r_2$ and
$\F_p$-bilinear forms $Q_1,\ldots,Q_r$ on $W_1\times W_2$
such that
$P_{2^9h}=\{(a,b)\in G_{k+1}^2\vert \text{\emph{rk} } M_{a,b}\leq 2^9 h\}$ contains the set
\begin{equation}
\label{eq:bilinearSet}
\{(x,y)\in W_1\times W_2\mid Q_1(x,y)=\cdots=Q_r(x,y)=0\}
\end{equation}
and
$\max(r,r_1,r_2)\leq O(\log^{80}\delta^{-1})$.
\end{theorem}
Applied with $\delta=q^{-n^c}$, this
means that the codimensions  should be $O(n^{O(c)})$.

The reason why sets of the form \eqref{eq:bilinearSet} are so desirable for us is the following lemma.
\begin{lemma}
\label{lm:isotropic}
Let $W$ be an $\F_p$-vector space of dimension $n$, and $Q_1,\ldots,Q_r$ be quadratic forms on $W$.
Then the set of isotropic vectors 
\begin{equation}
\label{eq:isotropic}
X=\{x\in W\vert Q_1(x)=\cdots=Q_r(x)=0\}
\end{equation}
contains at least $(1-p^{-1/2})p^{n-2r(r+1)}$
elements.
\end{lemma}
More compactly, we can write $\abs{X}\gg p^{n-O(r^2)}$.
We now prove Lemma \ref{lm:isotropic}. We introduce the 
averaging
notation
$\E_{x\in W} =\frac{1}{\abs{W}}\sum_{x\in W}$.
\begin{proof}
The density $\abs{X}/\abs{W}$ of isotropic vectors is given by
\begin{equation}
\label{eq:densityIsot}
\E_{x\in W}\E_{t_1,\ldots,t_r\in\F_p}\omega^{\sum_i t_iQ_i(x)}
=\E_{t_1,\ldots,t_r}\E_{x\in W}\omega^{\sum_i t_iQ_i(x)}
\end{equation}
where $\omega = e^{2 \pi i /p}$. 
Let $m\leq n$ be a parameter to be determined later (in terms or $r$).
Now if a quadratic form $Q$ on $W\times W$ has rank at least $m$, we can
see that 
$$
\abs{\E_{x\in W}\omega^{Q(x)}}\leq p^{-m/2}
$$
by squaring this expectation (see Lemma \ref{gaussSums}).
Thus if for any nonzero $(t_1,\ldots,t_r)$, the rank
of $\sum_i t_iQ_i$ is at least $m$, we see from equation \eqref{eq:densityIsot} that 
the density of isotropic vectors is at least
$p^{-r}-p^{-m/2}$.
Otherwise, there exists
a form $Q_i$ such that
$Q_i=\sum_{j\neq i} t_j Q_j + R$ with $\rk R <m$; without
loss of generality, suppose $i=r$.
Let $W'$ be the kernel of $R$, a subspace of codimension less than $m$. Then the set
\begin{equation}
\label{eq:isotropic2}
X'=\{x\in W'\vert Q_1(x)=\cdots=Q_{r-1}(x)=0\}
\end{equation}
is a subset $X$ and we will now count isotropic vectors in $X'$.
Thus incurring
a dimension loss of at most $m$, 
we reduce the number of quadratic forms by 1.
We iterate this process until
we get a family of quadratic forms for which any nontrivial linear combination has rank at least $m$ (or an empty family).
At that point, this is a family of at most $r$ forms
on a space of dimension at least $n-rm$.
Thus it must have at least
$$
p^{n-r(m+1)} -p^{n-rm-m/2}
$$
isotropic vectors.
Taking $m=2r+1$, we obtain the result.
\end{proof}
We will use this lemma in Section \ref{useBB} to obtain sets of
the form
\eqref{eq:monomes}
inside $P_{2^9h}$.

\subsection{Divisor bounds} We list some facts regarding the divisor function in $\Fq[t]$ which we will need in the sequel. Let $\tau(f)$ denote the number of monic divisors of $f \in \Fq[t]$. We first have

\begin{lemma} [{\cite[Lemma 8]{le1}}] \label{lem:divisorbound1}
If $\deg f = n >1$, then 
\[
\tau(f) \leq \exp \left( O_q \left( \frac{n}{\log n} \right) \right).
\]
Consequently, the number of monic irreducible factors of $f$ is $O_q \left( \frac{n}{\log n} \right)$. 
\end{lemma}
The next result is a bound for the second moment of $\tau$.
\begin{lemma}
\label{lem:divisorbound2}
We have
$$
\E_{\deg d=n}\tau(d)^2\leq 4n^3.
$$
\end{lemma}
\begin{proof}
We observe that
for any irreducible $P$
and any integer $k$, we have
$\tau(P^k)^2=(k +1)^2$.

Thus the Dirichlet series $D=\sum_{n=0}^{+\infty}\sum_{f\in A_n}\frac{\tau(f)^2}{\abs{f}^s}$ of the function $\tau ^2$ can be 
written as an Euler product as
\begin{equation}
\label{eq:eulerProd}
D=\prod_P\sum_{k=0}^{+\infty} (k+1)^2\abs{P}^{-ks}.
\end{equation}
Next we note 
the following relations between formal power series
$$
\sum_{k=0}^{+\infty} (k+1)^2x^k=\sum_{k=0}^{+\infty} (k+2)(k+1)x^k -\sum_{k=0}^{+\infty} (k+1)x^k
=2(1-x)^{-3}-(1-x)^{-2}=\frac{1+x}{(1-x)^3}
$$
so finally
\begin{equation}
\label{FPS}
\sum_{k=0}^{+\infty} (k+1)^2x^k=\frac{1-x^{2}}{(1-x)^4}.
\end{equation}
Combining equations \eqref{eq:eulerProd} and \eqref{FPS} yields
$$
D=\prod_P\frac{1-\abs{P}^{-2s}}{(1-\abs{P}^{-s})^4}.
$$
We can then express this Euler product in terms of
the zeta function of $\F_q[t]$. Letting $u=q^{-s}$ we obtain
$$
D=\zeta(s)^4/\zeta(2s)=(1-q^{1-2s})(1-q^{1-s})^{-4}=(1-qu^2)
(1-qu)^{-4}.
$$
This is a power series $S(u)$ in $u$, and
$S(u)=\sum_n a_nu^n=\sum_n \frac{S^{(n)}(0)}{n!}u^n$
where $a_n=\sum_{\deg d=n}\tau(d)^2$.
Now for $n\geq 3$, deriving $n$ times using Leibniz' formula, we find that
\begin{eqnarray*}
S^{(n)}(u) &=&(1-qu^2)q^n(4\times\cdots \times(n+3))(1-qu)^{-4-n}\\
&-&
2qunq^{n-1}(4\times \ldots \times(n+2))(1-qu)^{-3-n}\\
&-& 2q\binom{n}{2}q^{n-2}(4\times\ldots \times(n+1))(1-qu)^{-2-n}
\end{eqnarray*} 
Evaluating in $u=0$ gives 
$$
\frac{S^{(n)}(0)}{q^nn!}=(n+3)(n+2)(n+1)/6 -q^{-1}n(n+1)^2/6\leq 4n^3,
$$
where the left-hand side is exactly $\E_{\deg d=n}\tau(d)^2$. 
\end{proof}

\section{Character sum estimates} \label{sec:char}
In this section we prove the following.

\begin{theorem} \label{th:char1}
Let $l \geq 0, Q \in \Fq[t], \deg Q=m \geq 0$ and $\lambda$ be a character of $G_{l,Q}$. Then for any $d$, and $\epsilon >0$, we have

\begin{equation} \label{eq:char0}
\left| \sum_{f \in A_d} \mu(f) \lambda(f) \right| \ll_{\epsilon, q} q^{ \left( \frac{1}{2} + \epsilon \right) d + \epsilon (m+l)}
\end{equation}
\end{theorem}

\begin{proof}
First we assume that $\lambda$ is not principal. We will prove the following more precise bound
\begin{equation}
\label{eq:char2}
\left| \sum_{f \in A_d} \mu(f) \lambda(f) \right| \leq q^{\frac{d}{2} + \frac{d \log \log (m+l)}{\log (m+l)} + O_q \left( \frac{m+l}{\log^2 (m+l)} \right)} 
\end{equation}
Our method is a generalization of the proof of \cite[Theorem 2]{bll}. 

Like \cite{bll}, we deduce \eqref{eq:char2} from an estimate for $\log \calL(z, \lambda)$ near the circle $|z| = q^{-1/2}$, which is in turn deduced from an estimate for $\frac{\calL'(z,\lambda)}{\calL(z,\lambda)}$.

By taking the logarithmic derivatives of \eqref{eq:euler} and \eqref{eq:weil}, we have two different expressions for $\frac{\calL'(z,\lambda)}{\calL(z,\lambda)}$. On the one hand, we have
\begin{equation*}
 \frac{\calL'(z,\lambda)}{\calL(z,\lambda)} = \sum_{l=1}^\infty a_l z^{l-1}
\end{equation*}
where 
\begin{equation} \label{eq:al1}
 a_l = - \sum_{i=1}^{d(\lambda)} \alpha_i ^l
\end{equation}
according to \eqref{eq:weil}. 

On the other hand, according to \eqref{eq:euler}, we have
\begin{equation} \label{eq:al2}
 a_l = \sum_{\deg f=l} \Lambda(f) \lambda(f).
\end{equation}

From \eqref{eq:al1} we have
\begin{equation} \label{eq:al3}
 |a_l| \leq d(\lambda) q^{l/2}
\end{equation}
and from \eqref{eq:al2} we have
\begin{equation} \label{eq:al4}
 |a_l| \leq \sum_{\deg f=l} \Lambda (f) = q^l.
\end{equation}

Put $L=\lfloor 2 \log_q d(\lambda) \rfloor$. For $l > L$ we use the bound \eqref{eq:al3} and $l \leq L$ we use the bound \eqref{eq:al4}. Therefore, for any $z$, we have
\begin{equation} \label{eq:est1}
\left| \frac{\calL'(z,\chi)}{\calL(z,\chi)} \right| \leq \sum_{l=1}^{L} q^l |z|^{l-1} + \sum_{l=L+1}^\infty d(\lambda) q^{l/2} |z|^{l-1}.
\end{equation}

Let $0 < \epsilon <1/4$ be chosen later, $R = q^{-1/2 - \epsilon}$ and $w$ be arbitrary on the circle $|w| = R$. Integrating \eqref{eq:est1} along the line from 0 to $w$, and noting that $\calL(0, \lambda)=1$, we have

\begin{equation} \label{eq:est2}
 \left| \log \calL(w, \lambda) \right| \leq \sum_{l=1}^{L} \frac{(Rq)^l}{l} + \sum_{l=L+1}^\infty d(\lambda) \frac{(Rq^{1/2})^l}{l}.
\end{equation}

The second sum in \eqref{eq:est2} can be bounded by 
\begin{equation} \label{eq:crude1} 
\frac{d(\lambda)}{L} \sum_{l=L+1}^\infty (Rq^{1/2})^l \leq \frac{d(\lambda)}{L} R^{L} q^{\frac{L}{2}} \frac{1}{1-Rq^{1/2}} 
\ll \frac{d(\lambda)^2 R^{L}}{L} \frac{1}{1-Rq^{1/2}}.
\end{equation}
As for the first sum in \eqref{eq:est2}, we bound it crudely by
\begin{equation} \label{eq:crude2}
\sum_{l=1}^{L} (Rq)^l \leq (Rq)^{L} \sum_{k=0}^\infty (Rq)^{-k} \le \frac{d(\lambda)^2 R^{L}}{1 - (qR)^{-1}} \ll_q d(\lambda)^2 R^{L}
\end{equation}
since $qR \geq q^{1/4}$.
By combining \eqref{eq:crude1} and \eqref{eq:crude2}, we have
\[
\left| \log \calL(w, \lambda) \right| \ll_q d(\lambda)^2 R^L \left( 1 + \frac{1}{L(1-Rq^{1/2)}} \right).
\]
Hence,
\begin{equation} \label{eq:l-2}
\left| \frac{1}{\calL(w, \lambda)} \right| \leq \exp \left( O_q \left( d(\lambda)^2 R^L \left( 1 + \frac{1}{L(1-Rq^{1/2)}} \right) \right) \right).
\end{equation}
Let $C_R$ be the circle $|w|=R = q^{-1/2 - \epsilon}$. From \eqref{eq:euler2} we see that
\begin{eqnarray}
\left| \sum_{f \in A_d} \lambda(f) \mu(f) \right| &=& \left| \frac{1}{2 \pi i} \int_{C_R} \frac{1}{\calL(w, \chi)} w^{-d-1} dw \right| \nonumber \\
& \leq & \max_{C_R} \left| \frac{1}{\calL(w, \lambda)} \right| R^{-d} \nonumber \\
& \leq &  q^{ d( 1/2 + \epsilon) + O_q \left( d(\lambda)^{1-2 \epsilon} \left( 1  + \frac{1}{ \epsilon \log d(\lambda)} \right)   \right)} \label{eq:char3}.
\end{eqnarray}
We now make the choice $\epsilon = \frac{\log \log d(\lambda)}{\log d(\lambda)}$. Recalling that $d(\lambda) \leq l+m -1$, \eqref{eq:char2} follows. The bound \eqref{eq:char2} is stronger than \eqref{eq:char0} when $\frac{\log \log (l+m)}{\log (l+m)}$ is greater than the $\epsilon$ in \eqref{eq:char0}. For the finitely many exceptional pairs $(m,l)$, \eqref{eq:char0} follows from \eqref{eq:char3} (with the same $\epsilon$).

We now consider the case where $\lambda$ is principal. From \eqref{eq:principal}, on the circle $|z| = q^{-1/2}$, we have
\begin{eqnarray}
\left| \frac{1}{\calL(z, \lambda)} \right| &=& | 1-qz | \prod_{P \in \calI, P|Q} \left| 1 - z^{\deg P} \right|^{-1} \nonumber \\
& \ll & \prod_{P \in \calI, P|Q} \left( 1 - q^{-\deg P/2} \right)^{-1} \nonumber \\
& \leq & \prod_{P \in \calI, P|Q} \left( 1 - q^{-1/2} \right)^{-1}  \nonumber \\
&=& (1-q^{-1/2})^{-k} \leq q^{O_q \left( \frac{m}{\log m} \right)} \label{eq:l-bound}
\end{eqnarray}
where $k$ is the number of monic irreducible factors of $Q$ and \eqref{eq:l-bound} follows from Lemma \ref{lem:divisorbound1}. 
Integrating $z^{-d-1} \frac{1}{\calL(z, \lambda)} $ along the circle $|z|=q^{-1/2}$ and using \eqref{eq:l-bound}, we see that 
\begin{equation} \label{eq:char-principal}
\sum_{f \in A_n, (f,Q)=1} \mu(f) \ll q^{\frac{d}{2} + O_q \left( \frac{m}{\log m} \right)}
\end{equation}
from which \eqref{eq:char0} follows.
\end{proof}
We remark that \eqref{eq:principal} readily gives a formula for $\sum_{f \in A_n, (f,Q)=1} \mu(f)$ but it is not immediate to derive \eqref{eq:char-principal} from this formula.

\section{Exponential sum estimates} \label{sec:exponential}
We say a function $F: M \rightarrow \C$ is $R_{l,Q}$-periodic if it is constant on each equivalence class of $R_{l,Q}$. In other words, $F$ is 
$R_{l,Q}$-periodic if $F(f)$ depends only on the residue class of $f$ modulo $Q$ and the first $l$ coefficients of $f$. We say $F$ is 1-bounded if $|F(f)| \leq 1$ for any $f \in M$. First we show that $\mu$ is orthogonal to $R_{l,Q}$-periodic functions by adapting the argument of \cite[Proposition 3.2]{gt}.

\begin{proposition} \label{prop:periodic}
Suppose $\deg Q = m$. For any $R_{l,Q}$-periodic and 1-bounded function $F: M \rightarrow \C$ and $\epsilon >0$, we have 
\[
\sum_{f \in A_n} F(f) \mu(f) \ll_{\epsilon,q} q^{(1/2 + \epsilon) (n+m+l)}
\]
where the bound is uniform in $F$.
\end{proposition}

\begin{proof}
We first consider the case where $F(f) = 0$ whenever $(f,Q) \neq 1$. This means that $F$ is a function on $G_{l,Q}$. Let 
$K = |G_{l,Q}| = q^{l} \phi(Q) \leq q^{l+m}$ and $\lambda_1, \cdots, \lambda_K$ be the characters of $G_{l,Q}$. Define the Fourier coefficients of $F$ by
\[
\widehat{F}(\lambda) = \E_{f \in G_{l,Q}} F(f) \overline{\lambda(f)}
\]
for any character $\lambda$ of $G_{l,Q}$. Then $F(f) = \sum_{i=1}^K \widehat{F}(\lambda_i) \lambda_i (f)$. Plancherel's formula implies
\begin{equation} \label{eq:plancherel}
\sum_{i=1}^K \left| \widehat{F}(\lambda_i) \right|^2  = \E_{f \in G_{l,Q}} |F(f)|^2 \leq 1.
\end{equation}
We have
\begin{eqnarray}
\left| \sum_{f \in A_n} F(f) \mu(f) \right| &=& \left| \sum_{i=1}^K \widehat{F}(\lambda_i) \sum_{f \in A_n}\lambda_i (f) \mu(f) \right| \nonumber \\
&\ll_{\epsilon,q} &  q^{n/2 + \epsilon(n+l+m)} \sum_{i=1}^K \left| \widehat{F}(\lambda_i) \right| \label{eq:periodic1} \\
&\leq & q^{n/2 + \epsilon(d+l+ m)} K^{1/2} \label{eq:periodic2} \\
&\leq & q^{n/2 + (l+m)/2 + \epsilon(n+l+m)}. \nonumber 
\end{eqnarray}
Here \eqref{eq:periodic1} follows from Theorem \ref{th:char1} and \eqref{eq:periodic2} follows from the Cauchy-Schwarz inequality and \eqref{eq:plancherel}.

Next we consider the general case where $F(f)$ is not necessarily 0 when $(f,Q)=1$. If $f$ is square-free, $(f,Q)=D$, we can write $f = Dg$ where $g$ is square-free and $(g, Q)=1$. Hence
\begin{eqnarray}
\sum_{f \in A_n} F(f) \mu(f) &=& \sum_{\substack{D \in M, D|Q,\\ D \textup{ square-free}}} \sum_{\substack{\deg g = n - \deg D,\\ g \textup{ square-free}}} F(Dg) \mu(Dg) 1_{(g, Q) =1} \nonumber \\
&=& \sum_{D \in M, D|Q} \mu(D) \sum_{\substack{\deg g = n - \deg D,\\ g \textup{ square-free}}} F(Dg) \mu(g) 1_{(g, Q) =1}
\end{eqnarray}
Now the function $g \mapsto F(Dg) \mu(g) 1_{(g, Q) = 1 }$ is $R_{l,Q}$-periodic, and vanishes on elements of $M$ that are not
coprime to $Q$. By the above, we infer that 
$$
\sum_{\substack{\deg g = n - \deg D,\\ g \textup{ square-free}}} F(Dg) \mu(g) 1_{(g, Q) =1}\ll_{\epsilon,q} q^{\frac{n-\deg D}{2}+\frac{l+m}{2}+\epsilon(n+m+l)}
$$
for any $\epsilon >0$.
Furthermore, still for any $\epsilon >0$, we observe that
$$
\sum_{D\mid Q}q^{-(\deg D)/2}\leq \tau(Q) \ll_{\epsilon,q} \abs{Q}^{\epsilon}=q^{\epsilon m}
$$
by Lemma \ref{lem:divisorbound1}.
This completes the proof.
\end{proof}
We will now use Proposition \ref{prop:periodic} 
and the ideas outlined at the beginning of Section \ref{sec:hayes}
to prove the following exponential sum estimate.

\begin{theorem}
\label{trm:torusforms}
Given any $\epsilon >0$, for all $\alpha \in \T$ and $n$, we have
\begin{equation} \label{eq:mu-linear}
\sum_{f \in A_n} \mu(f) e( \alpha f ) \ll_{\epsilon,q} q^{(3/4 + \epsilon) n}
\end{equation}
and
\begin{equation} \label{eq:mu-linearG}
\sum_{f \in G_n} \mu(f) e( \alpha f ) \ll_{\epsilon,q} q^{(3/4 + \epsilon) n}.
\end{equation}
\end{theorem}
The first bound implies the second bound,
because
$$
\sum_{f\in G_n} \mu(f) e( \alpha f ) =
\sum_{c\in\F_q^*}\sum_{k=0}^{n-1}  \sum_{f\in A_k}\mu(f)e(\alpha cf) 
$$
so we only need to prove the bound \eqref{eq:mu-linear}.
It is easy to see that any linear form on $G_n$ can be written
as $f\mapsto (\alpha f)_{-1}$ (i.e., the coefficient of $t^{-1}$ in $\alpha f$) for some $\alpha\in\T$. Thus
Theorem \ref{trm:linforms} follows from Theorem \ref{trm:torusforms}.





\begin{proof}
By Dirichlet's approximation theorem, we can find $a, g \in \Fq[t], g \neq 0, \deg g \leq \lfloor \frac{n}{2} \rfloor$ such that 
$\left| \alpha - \frac{a}{g} \right| < \frac{1}{q^{\lfloor \frac{n}{2} \rfloor} |g|}$. Put $\beta = \alpha - \frac{a}{g}$. Then
\[
\sum_{f \in A_n} \mu(f) e( \alpha f ) = \sum_{f \in A_n} \mu(f) e \left( \frac{af}{g} \right) e(\beta f).
\]
Since $|\beta| < q^{-\lfloor \frac{n}{2} \rfloor - \deg g}$, we see that $e(\beta f)$ depends only on the first 
$n-\lfloor \frac{n}{2} \rfloor - \deg g$ coefficients of $f$.  Also, $e \left( \frac{af}{g} \right)$ depends only on the residue class of $f$ modulo $g$. Applying Proposition \ref{prop:periodic} to 
$(l,Q) = (n-\lfloor \frac{n}{2} \rfloor - \deg g, g)$, for any $\epsilon >0$, we have
\[
\sum_{f \in A_n} \mu(f) e \left( \frac{af}{g} \right) e(\beta f) \ll_{\epsilon,q} q^{\frac{1+\epsilon}{2} (n + n-\lfloor \frac{n}{2} \rfloor -\deg g +\deg g)} = 
q^{\frac{1+\epsilon}{2} (2n - \lfloor \frac{n}{2} \rfloor)} \ll_{\epsilon,q} q^{(3/4 + \epsilon)n},
\]
as desired.
\end{proof}

As we show next, this implies that if a function is determined
by the values of a few linear forms, it does not correlate
with the M\"obius function.
\begin{corollary}
\label{fewforms}
Let $c>0$ be a constant. Let $F : \F_q^r\rightarrow\C$ be $1$-bounded and suppose $r \leq cn$. Let $\ell_1,\ldots,\ell_r$ be linear forms on $G_{n}$. Then for any $\epsilon >0$,
$$
\sum_{f\in G_n} \mu(f)F(\ell_1(f),\ldots,\ell_r(f))\ll_{\epsilon,q} q^{(3/4+c+\epsilon)n}.
$$
\end{corollary}
Obviously, this is interesting only if $c<1/4$.
\begin{proof}
Theorem \ref{trm:linforms} immediately implies that for any linear forms $\ell$ on $G_n$, we have
\begin{eqnarray} \label{eq:cor}
\sum_{f\in G_n} \mu(f) e_q(\ell(f)) \ll_{\epsilon,q} q^{(3/4+\epsilon)n}.
\end{eqnarray}

For any $\ba=(a_1,\ldots,a_r)\in\F_q^r$, 
let $V_\ba\leq G_n$ be the affine subspace defined by the equations $\ell_i(f)=a_i$
for $i\in [r]$.
Then one can write
\begin{equation}
\label{decompoSubspaces}
\sum_{f\in G_n}\mu(f)F(\ell_1(f),\ldots,\ell_r(f))
=\sum_{\ba \in\F_q^r}F(\ba)
\sum_{f\in V_\ba}
\mu(f).
\end{equation}
Now we observe that
$$
1_{V_\ba} (f)=\E_{ \bchi = (\chi_1,\ldots,\chi_r) \in \widehat{\F_q}^r}\prod_{i\in [r]}\chi_i(\ell_i(f)-a_i)
$$
so that
$$
\sum_{f\in V_\ba }\mu(f)=\E_{\bchi \in\widehat{\F_q}^r}\prod_{i\in [r]}\chi_i(-a_i)\sum_{f\in G_n}\mu(f)\prod_{i\in [r]}\chi_i(\ell_i(f))
$$
and by the triangle inequality
$$
\abs{\sum_{f\in V_\ba}\mu(f)}\leq \max_{\bchi\in\widehat{\F_q}^r}\abs{\sum_{f\in G_n}\mu(f)\prod_{i\in [r]}\chi_i(\ell_i(f))}
$$
Recall from Section \ref{sec:notation} that each $\chi_{i}$ is of the form $\chi_i (x) = e_q(t_i x)$, so that
$$
\prod_{i\in [r]}\chi_i(\ell_i(f))=e_q\left(\sum_{i=1}^r t_i \ell_i(f)\right)
$$
We then apply \eqref{eq:cor} to the linear form
$\ell=\sum_{i\in [r]}t_i\ell_i$. This shows that
$$
\abs{\sum_{f\in V_\ba}\mu(f)}\ll q^{(3/4 + \epsilon) n}.
$$
Plugging this bound in equation \eqref{decompoSubspaces} and
using the fact
that $\abs{F}\leq 1$, this gives the desired result.
\end{proof}

\section{Quadratic phases and Vaughan's identity} \label{sec:quad}
From now on, we suppose the field $\F_q$ we work with has characteristic $p>2$. Recall $q=p^s$ and $s\geq 1$.
\subsection{Quadratic phases}
We call \emph{quadratic form} on $\Fqn$ a homogenous polynomial of degree 2, that is, a map of the form
$
F(x)=x^TMx
$
where $M$ is a symmetric matrix. The corresponding (symmetric) bilinear form is the map
$$
B(x,y)=x^TMy.
$$
The \emph{rank} of $F$ is the
rank of the matrix $M$. It equals the codimension of the space $K$ 
of vectors $x$ such that the linear form $B_x$
defined by $B_x(y)=B(x,y)$ satisfies $B_x=0$.
A \emph{quadratic polynomial} is a polynomial of degree 2, that
is, a quadratic form plus a linear form. A
\emph{quadratic phase} is a map of the form
$\Phi(x)=\chi(P(x))$ for a quadratic polynomial $P$ and an additive character $\chi$.
Its \emph{rank}
is the rank of the corresponding quadratic form. Thanks to the following standard lemma, quadratic phases can be
classified, depending on their rank, into major arcs and
minor arcs, by analogy with the circle method.
\begin{lemma}[Gauss sums]
\label{gaussSums}
Let $\Phi(x)=\chi(P(x))$ be a quadratic phase of rank at least $r$.
Then
$$
\abs{\E_{x\in\F_q^n}\Phi(x)}\leq q^{-r/2}.
$$
\end{lemma}
Thus quadratic phases of low rank correspond to major arcs,
while the ones of high rank correspond to minor arcs.
\begin{proof}
We use the standard technique called Weyl differencing, consisting
of squaring the expectation to duplicate the variable.
We have
\begin{eqnarray*}
\abs{\E_{x\in\F_q^n}\Phi(x)}^2 &=&\E_{x,h}\Phi(x+h)\overline{\Phi(x)}\\
&=&\E_{x,h}\chi(P(x+h)-P(x))\\
&=&\E_{h}\chi(P(h))\E_x\chi(2B_h(x))
\end{eqnarray*}
where all variables range over $\F_q^n$.
Now if $h\notin K$, the form $2B_h$ is a nonzero linear form
(remember the characteristic $p$ is not 2), whence
$\E_{x\in\Fqn}\chi(2B_h(x))=\E_{x\in\F_q}\chi(x)=0$.
This implies that
$$
\abs{\E_{x\in\F_q^n}\Phi(x)}^2\leq \E_{h\in\Fqn}1_{h\in K}=q^{-r}
$$
and the claim follows.
\end{proof}
We now start the proof of Theorem 2.
Let $P$ be a quadratic polynomial on $G_n$ and $\Phi=\chi\circ P$ be a quadratic phase. We want to bound the sum
\[
\sum_{f\in G_n}\mu(f)\Phi(f).
\]
The general strategy
is the following. We first observe that when $\Phi$ is a quadratic phase
of rank at most $cn$ with $c<1/4$, 
then Corollary \ref{fewforms} concludes: indeed, a quadratic
form of rank $r$ depends on $r$ linear forms only,
so a quadratic polynomial of rank $r$ depends on $r+1$ linear forms at most. So we will show that in order for $\mu$
to correlate with a quadratic phase $\Phi$, the corresponding
quadratic form needs to be of small rank (major arcs).
This would imply that $\mu$ cannot correlate with a quadratic phase at all.

\subsection{Exploitation of Vaughan's identity}
\label{Vaughan}
We will show the following.
\begin{proposition}
\label{smallRank}
Let $\delta >0$. Suppose 
$\abs{\sum_{f\in G_n}\mu(f) \Phi(f)}\geq \delta q^n$.
Then 
at least one of the following two statements holds.
\begin{enumerate}
\item There exists $k\leq n/9$ such that for at least one
polynomial
$d$ of degree $k$, the quadratic polynomial on $G_{n-k}$
defined by
$$
w\mapsto P(dw)
$$
has rank at most $O(\log (n/\delta))$.
\item 
There exists $k\in [n/18,17n/18]$ such that
for at least $(\delta/n)^{O(1)}q^{2k}$ pairs of polynomials
$d,d'$ of degree $k$, the quadratic polynomial on $G_{n-k}$
defined by
$$
w\mapsto P(d'w)-P(dw)
$$
has rank at most $O(\log (n/\delta))$.
\end{enumerate}
\end{proposition}
Before proving this proposition, we underline that for any
$d\in G_{k+1}$, we see the map $w\mapsto dw$ as a linear map from
$G_{n-k}$ to $G_n$
which allows one to see $w\mapsto P(dw)$ as a quadratic polynomial.

We now start proving the proposition.
The first tool we need is Vaughan's identity, which reads
$$
\mu(f)=-\sum_{\substack{ab\mid f\\\deg a\leq u,\deg b\leq v}}\mu(a)\mu(b)+\sum_{\substack{ab\mid f\\\deg a> u,\deg b> v}}\mu(a)\mu(b)
$$
where the sum is over monic polynomials $a$ and $b$, and $u=v=n/18$ (though in general they can be chosen arbitrarily).
We shall adopt the notational convention that
$N=q^n,U=q^u$ and so on. Moreover, for $f\in\F_q[t]$, recall the notation
$\abs{f}=q^{\deg f}$.
Vaughan's identity implies that
\begin{equation}
\label{eq:T1T2}
\sum_{f \in G_n} \mu(f)\Phi(f)=-T_1+T_2
\end{equation}
where
\begin{equation}
\label{eq:T1}
T_1=\sum_{\abs{d}\leq UV}a_d\sum_{w \in G_{n -\deg d}}\Phi(dw)
\end{equation}
and
\begin{equation}
\label{eq:T2}
T_2=\sum_{V\leq \abs{d} \leq N/U}b_d\sum_{w \in G_{n-\deg d}}
\mu(w)\Phi(dw)
\end{equation}
are called type I and type II sums respectively.
The sums over $d$ are over monic polynomials.
The coefficients $a_d$ are unimportant and all we need to know is that $\max(\abs{a_d},\abs{b_d})\leq\tau(d)$.
In the type I sum, we have made the change of variables $d=ab,w=f/d$, while in the other one we wrote $w=\text{lc}(f)b,d=f/w$,
where lc stands for leading coefficient, so that $d$
is monic.
The splitting into two sums yields the following dichotomy,
which we will use to prove Proposition \ref{smallRank}.
\begin{proposition}
\label{T1ouT2}
Make the same hypothesis as in Proposition \ref{smallRank}.
Then either
 there exists $k\leq n/9$ so that
\begin{equation}
\label{eq:smallk}
\E_{d \in A_k}\abs{\E_{w \in G_{n-k}}\Phi(dw)}^2\geq \delta^2/(16n^5)
\end{equation}
or there is a $k\in [n/18,17n/18]$ such that
\begin{equation}
\label{eq:largek}
\E_{w,w'\in G_{n-k}}\E_{d,d'\in A_k}\Phi(dw) \overline{ \Phi(dw')\Phi(d'w) }\Phi(d'w')\geq \delta^4/(256n^{10}).
\end{equation}
\end{proposition}
\begin{proof}
If $\left| \sum_{f \in G_n} \mu(f) \Phi(f) \right| \geq \delta N$, the decomposition
\eqref{eq:T1T2} implies that
either $\abs{T_1}\geq \delta N/2$
or
$\abs{T_2}\geq \delta N/2$.
Suppose first
$\abs{T_1}\geq \delta N/2$.
On the other hand, using the triangle inequality and equation \eqref{eq:T1}, we bound $T_1$ by
\begin{eqnarray*}
\abs{T_1} &\leq& \sum_{k\leq u+v}\sum_{d \in A_k}\tau(d)\frac{N}{\abs{d}}\abs{\E_{w \in G_{n-k}} \Phi(dw)} \\
&\leq& n\max_{k\leq u+v}
\frac{N}{K}\sum_{d \in A_k}\tau(d)\abs{\E_{w \in G_{n-k}}\Phi(dw)}.
\end{eqnarray*}
Fix a $k\leq u+v=n/9$ that realises the maximum in the line above.
The Cauchy-Schwarz inequality then yields
$$
\abs{T_1}^2/N^2\leq n^2\, \left( \E_{d \in A_k}\,\tau^2(d)\, \right) \left( \E_{d \in A_k}\abs{\E_{w \in G_{n-k}}\Phi(dw)}^2 \right).
$$
Now Lemma \ref{lem:divisorbound2} ensures that
$$
\E_{d \in A_k}\,\tau^2(d)\leq 4 k^3\leq 4n^3,
$$
so  we can affirm that
$$\delta^2N^2/4\leq \abs{T_1}^2\leq 4n^{5}N^2\E_{d \in A_k}\abs{\E_{w \in G_{n-k}} \Phi(dw)}^2.$$
This means that
$$
\E_{d \in A_k}\abs{\E_{w \in G_{n-k}} \Phi(dw)}^2\geq \delta^2/(16n^5)
$$
which proves equation \eqref{eq:smallk}.

Let us now suppose that $\abs{T_2}\geq\delta N/2$. 
Using the triangle inequality and equation \eqref{eq:T2}, we have
\begin{eqnarray*}
\abs{T_2} &\leq & \sum_{V\leq\abs{d}\leq N/U}\tau(d)\abs{\sum_{w \in G_{n-k}} \mu(w)\Phi(dw)}\\
&\leq & nN
\max_{v\leq k\leq n-u}\E_{d \in A_k}\tau(d)\abs{\E_{w \in G_{n-k}}\,\mu(w)\Phi(dw)}.
\end{eqnarray*}
We again fix a $k$ (this time $k\in[n/18,17n/18]$) that realises the maximum and apply Cauchy-Schwarz together with Lemma \ref{lem:divisorbound2}, obtaining
$$
\abs{T_2}^2/N^2\leq 4n^5\, \E_{d\in A_k}\E_{w, w' \in G_{n-k}}\mu(w)
\mu(w')\Phi(dw) \overline{\Phi(dw')}.
$$
This implies that
$$
\E_{w,w'\in G_{n-k}}\,\mu(w)
\mu(w')\E_{d\in A_k}\Phi(dw) \overline{\Phi(dw')} \geq \delta^2/(16n^5)
$$
and again we apply Cauchy-Schwarz to eliminate $\mu$, which yields
$$
\E_{w,w'\in G_{n-k}}\E_{d,d'\in A_k}\Phi(dw) \overline{ \Phi(dw')\Phi(d'w) } \Phi(d'w')\geq \delta^4/(256n^{10}).
$$
This is the content of clause \eqref{eq:largek}, 
so the proof of Proposition \ref{smallRank} is complete.
\end{proof}
We now derive Proposition \ref{smallRank} using Proposition \ref{T1ouT2}.
Suppose first that
equation \eqref{eq:smallk} holds, so
 there is $k\leq n/9$ such that
\begin{equation}
\label{eq:smallExp}
\E_{d\in A_k}\abs{\E_{w \in G_{n-k}}\Phi(dw)}^2\geq \delta',
\end{equation}
where $\delta'=\delta^2/(16n^5)$.
Equation \eqref{eq:smallExp} implies that there exists $d\in A_k$ such that $$\abs{\E_{w \in G_{n-k}}\Phi(dw)}^2\geq \delta'.$$
Fix such a $d\in A_k$.  Lemma \ref{gaussSums} then implies that the 
quadratic polynomial $w\mapsto P(dw)$ has rank at most
$\log_q(\delta'^{-1})=O(\log(n/\delta))$.
This corresponds exactly to the first statement of Proposition \ref{smallRank}.


Suppose instead 
that equation \eqref{eq:largek} holds. Then
we have $k\in [n/18,17n/18]$ such that
$$
\E_{w,w'\in G_{n-k}}\E_{d,d'\in A_k}\Phi(dw) \overline{ \Phi(dw')\Phi(d'w) } \Phi(d'w')\geq \delta',
$$
where $\delta'=\delta^4/(256n^{10})$.
The triangle inequality ensures that
$$
\E_{d,d' \in A_k}\abs{\E_{w\in G_{n-k}}\chi(P(dw)-P(d'w))}\geq \delta'.
$$
In particular, for a proportion at least $\delta'/2$
of pairs of monic polynomials $d,d'$ of degree $k$,
we have
$$
\abs{\E_{w\in G_{n-k}} \chi (P(dw)-P(d'w))}\geq \delta'/2
$$
which implies that the rank of
$w\mapsto P(dw)-P(d'w)$ is at most $-\log_q (\delta'/2)=O(\log(n/\delta))$.
This is precisely the second part of Proposition \ref{smallRank}.
So in every case, Proposition \ref{smallRank} holds.

\section{Using the polylogarithmic bilinear Bogolyubov theorem} \label{useBB}
Let $c>0$ be a constant to be determined later and let $\delta=q^{-n^c}$.
To prove Theorem \ref{trm:quadforms}, it suffices to show that
$\abs{\sum_{f\in G_n}\mu(f) \Phi(f)}< \delta q^n$
for $n$ large enough.
For the sake of contradiction,
suppose instead that there exists 
an unbounded set $Z$ of integers $n$ such that
\begin{equation}
\label{eq:contradiction}
\abs{\sum_{f\in G_n}\mu(f) \Phi(f)}\geq \delta q^n
\end{equation}
 whenever $n\in Z$.
We then apply Proposition \ref{smallRank}.
Suppose the first statement holds.
Write $P(f)=B(f,f)$ for some bilinear form $B(x,y)$
on
$\F_q^n\times\F_q^n$ (we may omit the linear part of $P$ as it modifies the rank by at most 1). 
Then we know that
the form $R_d : w\mapsto P(dw)$ on $G_{n-k}$ has rank at most $O(n^c)$ for at least one $d$ of some degree $0\leq k\leq n/9$. Now the rank of the quadratic form $R_d$ is simply the rank
of the bilinear form $B$ restricted to the subspace $dG_{n-k}\subset G_n$ of codimension $k$. Thus the rank of $R_d$ is at least $\rk B-2 k$, which implies that $\rk B\leq 2n/9 + O(n^c)$. 
If $n\in Z$ is large enough (remember $Z$ is unbounded) this is less than $c'n$ for some $c'<1/4$. 
Then Corollary \ref{fewforms} brings the desired contradiction.


Now let us suppose that
the second case of Proposition \ref{smallRank} holds.
Let
 $n/18\leq k\leq 17n/18$ be the parameter returned by this proposition.
Then the set
 $$
Y=\{(d,d')\in A_k^2\mid w\mapsto P(dw)-P(d'w)\text{ has rank at most }  O(n^c)\}
$$
has size at least $q^{2k+2-O(n^c)}$.
Note that for $d, d' \in G_{k+1}$, 
\[
P(dw)-P(d'w)=B((d-d')w,(d+d')w)
\]
is a quadratic polynomial in $w\in G_{n-k}$. 
For $a, b \in G_{k+1}$, let $B_{a,b}$ be the symmetric bilinear form on $\F_q^{n-k}\times\F_q^{n-k}$ (identified with $G_{n-k} \times G_{n-k}$) defined by
$B_{a,b}(x,y)=(B(ax,by)+B(ay,bx))/2$.
Thus we have a set
\[
X=\{(a,b)\in G_{k+1}\times G_{k+1} \mid \rk B_{a,b}\leq O(n^c)\}
\]
of density at least $\eta=q^{-O(n^c)}$ in $G_{k+1} \times G_{k+1}$.
As discussed in  Section \ref{sec:bogo}, we would like to replace the large set $X$ by a more structured
set, namely the zero set of a (not too large) family of bilinear forms, at the cost of slightly worsening the bounds on the rank.
Theorem \ref{conj:polyLog},
an application of the bilinear Bogolyubov theorem from \cite{LovettHosseini}, precisely implies that
$$
X'=\{(a,b)\in G_{k+1}\times G_{k+1} \mid \rk B_{a,b}\leq O(n^c)\}
$$
contains a set of the form
$$
Y=\{(a,b)\in W_1\times W_2\mid F_1(a,b)=\ldots=F_r(a,b)=0\}
$$
where $W_1,W_2$ are $\F_p$-subspaces of $G_{k+1}$ (itself seen as an $\F_p$-vector space of dimension $s(k+1)=O(k)$) and $\max (\textrm{codim} W_1, \textrm{codim} W_2, r)= O(\log^{80}\eta^{-1})=O(n^{80c})$.

Now take $\epsilon =1/10$ and consider a set of indices $$I=\{0=i_1<i_2<\cdots < i_m=\lfloor k-\epsilon k\rfloor\} \subset [0,k-\epsilon k]$$ such that $i_{j+1}-i_j < (n-k)/2$ for any $j$ and $m=O(1)$.
Such a set exists because $n-k\geq n/18\geq k/18$. Consider $W=W_1\cap W_2\cap G_{\epsilon k}$, an $\F_p$-vector space of dimension at least $\epsilon s k-O(n^{80c})$.
Consider the $\F_p$-quadratic forms on $W$ given by $F_l^{i,j}(w)=F_l(t^iw,t^jw)$ for any $l\in [r]$ and $i,j\in I$, where the map $w\mapsto t^i w$ is identified with the corresponding $\F_p$-linear map between the vectors of coefficients. This
is still a family of at most $O(n^{80c})$ bilinear forms.
Thus we can find at least $\Omega(p^{\epsilon s k - O(n^{160c})})$ common isotropic vectors in $W$ to these forms, thanks to Lemma \ref{lm:isotropic}. 
We take $c$ sufficiently small such that the exponent of $n$
in the last equation is less than $1$ ($c=1/161$ is good enough).
Then if $k$ (and thus $n$) is large enough, there is definitely at least one nonzero polynomial $w$ of degree at most $\epsilon k$
such that 
$
F_l(t^iw,t^jw) = 0
$ for all $i, j \in I$ and $ l \in [r]$. Consequently, $\rk B_{t^iw,t^jw}\leq \kappa=O(n^c)$ for
all $i,j \in I$.

Consider the (symmetric) matrix $M$ of the $\F_q$-bilinear form $B$ restricted to 
the space of the multiples of $w$, written in the basis 
$(wt^i)_{0\leq i < n-\deg w}$. We call
the matrix element $B(wt^i,wt^j)$ the \emph{cell} $(i,j)$ of $M$. The rank of $B$ differs 
from the rank of $M$ by at most $2\epsilon n$, so it suffices to bound the rank of $M$.

Now let us examine the (symmetric) matrix $N_{i,j}$ of the quadratic form $B_{t^iw,t^jw}$ in the canonical basis  of $G_{n-k}$.

Observe that the map $w\mapsto t^i w$, seen as a
$\F_q$-linear map (between vectors of coefficients), transforms 
an element $t^j$ of the canonical basis of $G_{n-k}$ into a basis element $t^{i+j}w$. That means that its matrix in the canonical basis of $G_{n-k}$ and the basis $(wt^i)_{0\leq i < n-\deg w}$
is an $(n-\deg w)\times (n-k)$ matrix which we can write by block as
$$L_{t^iw}=\begin{pmatrix}
0 \\ 
I_{n-k} \\ 
0
\end{pmatrix} $$
where the central block is an identity block and the other blocks are $0$ blocks.

A submatrix of a matrix consisting of consecutive rows and columns is called a \emph{block}.
Next we observe that
$$2 N_{i,j}=L^T_{t^iw} M L_{t^j w}+L^T_{t^jw} M L_{t^i w}$$
which makes it easy to see that
$N_{i,j}$ is the
symmetric part of the 
 $(n-k)\times (n-k)$ block of $M$ whose top-left corner is the
 $(i,j)$ cell of $M$. Write $M_{i,j}$ for this block.

We remark that if $i=j$, this block is a diagonal block of a
symmetric matrix, hence a symmetric matrix, so it must have small rank itself. Hence, the matrix $M$ contains a number of large diagonal blocks $M_{i,i}$ which have small rank.
To bound the rank of $M$, it suffices
to bound the ranks of all submatrices $M_{i,j}$ for $(i,j)\in I^2$.
Indeed, the matrix $M$ being covered by these submatrices,
we have the bound
$$
\rk M\leq \sum_{(i,j)\in I^2}\rk M_{i,j}\leq \abs{I}^2\max_{(i,j)\in I^2}\rk M_{i,j}.
$$
The cardinality $\abs{I}$ being bounded, bounding the
ranks of these blocks suffices to bound $\rk M$.
We now prove by induction on $\ell-m$ that
$M_{i_\ell,i_m}$ has small rank, namely at most $5^{\ell-m}\kappa$. Because $M_{i_\ell,i_m}=M^T_{i_m,i_\ell}$, it suffices to prove it in the case $\ell\geq m$.
When $\ell-m=0$, as we have already seen, the corresponding block is diagonal and of rank at most $\kappa$. 
We now suppose that for some $\ell\geq m$ we already know that
$\rk M_{i_\ell,i_m}\leq 5^{\ell-m}\kappa$ and we inspect $M_{i_{\ell+1},i_{m}}$.
The reader can follow the
proof on Figure \ref{covering}.
\begin{figure}
\begin{pspicture}(-1,-1)(11,12)
\psline[linestyle=dotted](-0.5,9)(1,9)
\rput(-0.7,9.2){$i_\ell$}
\psline[linestyle=dotted](1,9)(1,10.5)
\rput(1,10.7){$i_m$}
\psframe(0.5,0)(10.5,10)
\psframe(1,9)(5,5)
\psframe(2.5,7.5)(6.5,3.5)
\psline[linestyle=dotted](1,5)(-0.5,5)
\rput(-0.9,5.2){$i_\ell + n-k$}
\psline[linestyle=dotted](5,9)(5,10.5)
\rput(5,10.7){$i_m+n-k$}
\psframe[linestyle=dashed](1,7.5)(5,3.5)
\psframe[linestyle=dotted](3.5,7.5)(5,6)
\psline[linestyle=dotted](1,7.5)(-0.5,7.5)
\psline[linestyle=dotted](2.5,7.5)(2.5,10.5)
\rput(2.5,10.7){$i_{m+1}$}
\rput(-0.5,7.7){$i_{\ell+1}$}
\rput(-1,3.7){$i_{\ell+1}+n-k$}
\psline[linestyle=dotted](-0.5,3.5)(1,3.5)
\psline[linestyle=dotted](6.5,7.5)(6.5,10.1)
\rput(6.5,10.3){$i_{m+1}+n-k$}
\rput(1.5,6.375){$A$}
\rput(3.2,6.375){$B$}
\rput(4.2,6.675){$C'$}
\rput(1.5,4.375){$C$}
\rput(3.5,4.375){$D$}
\psframe[linestyle=dashed](6.5,9)(2.5,5)
\end{pspicture}
\caption{Covering $M$ by submatrices and moving away from the diagonal}
\label{covering}
\end{figure}
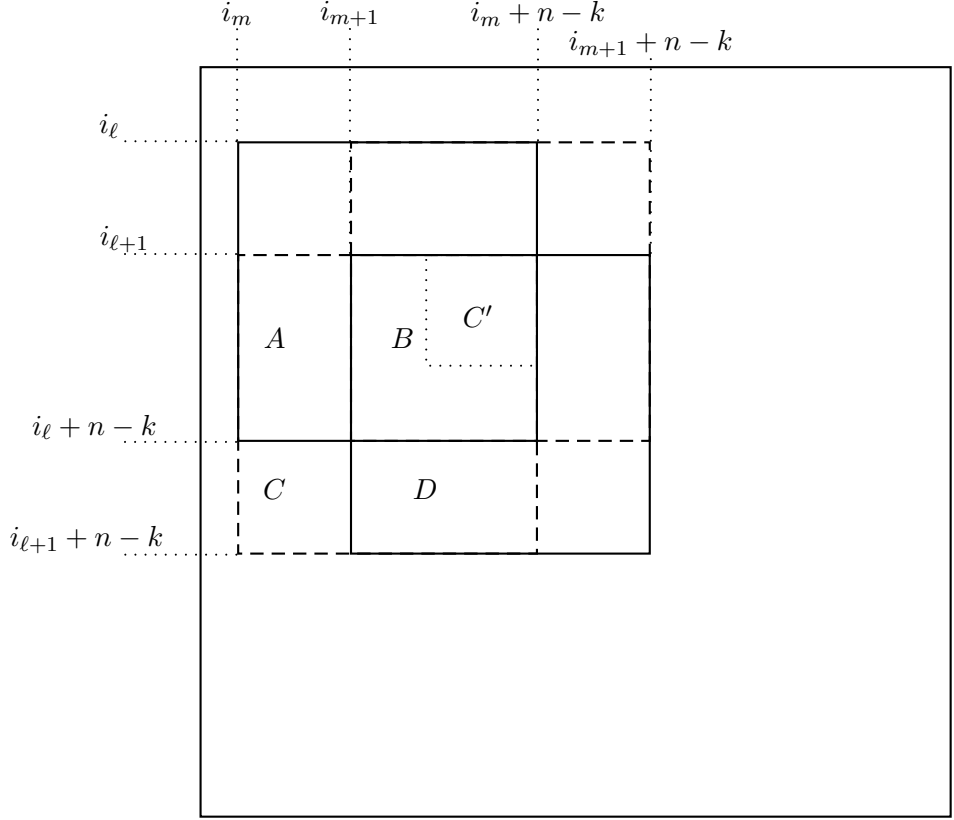

In Figure \ref{covering} the dotted $(n-k)\times (n-k)$ block $M_{i_{\ell+1},i_m}=E$ is
made of the four blocks $A,B,C,D$, and it
is known to have a symmetric part of small rank. 
On the other hand, $A,B$ and $D$ are already known to have rank
at most $5^{\ell-m}\kappa$, because they are submatrices of $M_{i_\ell,i_m}$
and $M_{i_{\ell+1},i_{m+1}}$ respectively. Now the symmetric part $E+E^T$ admits as 
bottom-left square block of the size of $C$ the matrix
$C+C'^T$ where $C'$ is the top-right block of $B$ (here it is crucial that $i_{\ell+1}-i_{\ell}<(n-k)/2$).
As a submatrix of a matrix of small rank,
$C+C'^T$ must have small rank. But $C'$ has small rank itself
as a submatrix of $B$, whence it follows that $C=(C+C'^T)-C'^T$ has small rank,
namely a rank at most $2\cdot 5^{\ell-m}\kappa$.
Hence
$$\rk M_{i_{\ell+1},i_{m}}=\rk E\leq \rk A + \rk B +\rk C+\rk D\leq 5^{\ell+1-m}\kappa .$$
This completes the induction proof,
and implies that $\rk M=O(\kappa)=O(n^c)$.

Finally, as already noted, the rank of $B$ is at most the rank of $M$ plus $2\epsilon n$. In particular, given that $2\epsilon=1/5$, it is surely less than $c'n$ for some $c'<n/4$, if $n\in Z$ is large enough.
Again invoking Corollary \ref{fewforms}, we obtain the desired contradiction with the hypothesis \eqref{eq:contradiction}.
This concludes the proof of Theorem \ref{trm:quadforms}. 


\section{The Hankel case}
\label{Hankel}
We prove Theorem \ref{trm:Hankel} and again we assume $p>2$.
If $\alpha = \sum_{j=-\infty}^m a_j t^j$ then the matrix of the quadratic form $f \mapsto (\alpha f^2)_{-1}$ in the canonical basis of $G_n$ is
\[
M=M(\alpha)=\begin{pmatrix}
a_{-1} & a_{-2} & \cdots & a_{-n} &  \\ 
a_{-2} & \iddots & \iddots & a_{-n-1} &  \\ 
\vdots & \iddots & \iddots & \vdots &  \\ 
a_{-n} & \alpha_{-n-1} &\cdots & a_{-2n+1} & 
\end{pmatrix}.
\]
We will follow the same strategy as in Sections \ref{Vaughan} and \ref{useBB} with $\Phi(f)=e(\alpha f^2 + \beta f)$. Suppose for a contradiction that, for arbitrarily large $n$, we have
\begin{equation}
\label{eq:contraHankel}
\sum_{f\in G_n}\mu(f)\Phi(f)>\delta q^n
\end{equation}
with $\delta=q^{-\epsilon' n}$ for some $\epsilon '>0$ to be decided later.
We apply Proposition \ref{smallRank}. 
We discard the first case of that proposition, because in that case the reasoning of Section \ref{useBB}
goes through without Conjecture \ref{conj:polyLog}. 
The parameter $\delta'=(\delta/n)^{O(1)}$
is still at least $q^{-\epsilon n}$ for some $\epsilon = O(\epsilon ')$, if $n$ is large enough.
Thus
we find a $k\in [n/18,17n/18]$
such that
for at least $q^{(2-\epsilon)(k+1)}$ pairs of polynomials
$(d,d')$ of degree $k$, the quadratic phase on $G_{n-k}$
defined by
$$
w\mapsto e(\alpha (d^2-d'^2)w^2)
$$
has rank at most $O(\epsilon n)$.
Write $d-d'=a$ and $d+d'=b$.
We infer
that
for at least $q^{(2-\epsilon)(k+1)}$ pairs of polynomials
$a,b$ of degree at most $k$, the quadratic phase 
$$
w\mapsto e(\alpha ab w^2)
$$
has rank at most $c \epsilon n$ for some constant $c = O(1)$.

With the notation of the previous section, the symmetric
matrix of that form is $$M_{a,b}=L_a^TM(\alpha)L_b=L_b^TM(\alpha)L_a=M(\alpha ab).$$
Thus compared to the general case,
$M_{a,b}$ is a product involving $M$ and not a sum of two
products, which makes it much easier to analyse. As in the proof of Theorem \ref{trm:quadforms}, we will show that $M$ has low rank by covering it by submatrices of low rank.

By Markov's inequality, there exists a set $X\subset G_{k+1}$ of size $q^{(1-\epsilon)(k+1)}/2$
such that for any $a\in X$, the set
\[
B_a:=\{b\in G_{k+1}\mid\rk M_{a,b}\leq c\epsilon n\}
\]
has size at least $q^{(1-\epsilon)(k+1)}/2$.

Let $\eta=2\epsilon$. For any $i\in\{0,\ldots,k-\eta k\}$ and $a\in X$, by the pigeonhole principle, there exist two distinct 
$b\neq b'$ in $B_a$ such that $f=b'-b=\sum_{m=i}^{i+\eta k}c_mt^m$ for some coefficients $c_m$.
Moreover, we have $\rk  M_{a,f}\leq 2c\epsilon n$.
Write $f=f_{a,i}$ to emphasize the dependence.
Fix $(i,j) \in\{0,\ldots,k-2\eta k\}^2$. Again the pigeonhole principle implies that there exist $a\neq a'\in X$
such that $g=a-a'\in\span (t^j,\ldots,t^{j+2\eta k})$ and
$f_{a,i}=f_{a',i}$.
If $f$ is this common value, we have $\rk  M_{g,f}=O(\epsilon n)$.
Observe that for such a pair $(g,f)$ we have
$$L_{g}=\begin{pmatrix}
0 \\ 
C_g \\ 
0
\end{pmatrix} $$
where the central block is a $(n-k+2\eta k)\times (n-k)$
matrix of rank $n-k$ and the other blocks are $0$ blocks.
The same holds for $L_f$, with a central block $C_f$.
So if $N$ is the $(n-k+2\eta k)\times (n-k+2\eta k)$
block of $M$ whose top-left cell is $(j,i)$, then
$M_{g,f}=C_g^TNC_f$
so that $\rk M_{g,f}\geq \rk N-4\eta k$.
As a result, $\rk N=O(\epsilon n)$.

Covering $M$ by a bounded number of blocks of
size $(n-k+2\eta k)\times (n-k+2\eta k)$,
we find that $\rk M=O(\epsilon n)$.
By taking $\epsilon$ small enough, the bound $O(\epsilon n)$
is constrained to be smaller than, say, $n/5$, for $n$
large enough. Thus if $\epsilon$ is small enough (that is, if $\epsilon'$ is small enough), we get a contradiction
between the hypothesis \eqref{eq:contraHankel} and Corollary \ref{fewforms}. Theorem \ref{trm:Hankel} follows.

It is possible to give an alternative proof of Theorem \ref{trm:Hankel} using the more traditional language of Diophantine properties of $\alpha$ and $\beta$. 
We have opted for the present proof since it shows parallels between the general case and the special case.   
%

\end{document}